\documentclass{svjour3}                     
\smartqed  
\usepackage{graphicx}
\usepackage{latexsym}
\usepackage{bbm}
\usepackage{color}
\usepackage{}
\usepackage{amsmath,amssymb,amsfonts}
\usepackage{subfigure}
\usepackage{array}
\usepackage{mathrsfs}
\usepackage{mathptmx}

\newtheorem{Theo}{Theorem}
\newtheorem{Rem}{Remark}
 \numberwithin{equation}{section}

\journalname{}
\begin{document}

\title{Unconditional superconvergence analysis of a linearized Crank-Nicolson Galerkin FEM for generalized Ginzburg-Landau equation
\thanks{This work was supported by NSF of
China (No. 11771163) and  China Postdoctoral Science Foundation (No. 2018M632791).
}}
\titlerunning{~~}

\author{Meng Li \and Dongyang Shi$^*$ \and Junjun Wang}

\institute{M. Li \and D. Shi \and J. Wang \at
              School of Mathematics and Statistics, \\ Zhengzhou University, Zhengzhou 450001, China \\
           \and
           D. Shi(Corresponding author) \at
             \email{shi\_dy@zzu.edu.cn} }

\date{Received: date / Accepted: date}

\date{Received: date / Accepted: date}

\maketitle

\begin{abstract}
In this paper, a linearized Crank-Nicolson Galerkin finite element method (FEM) for generalized Ginzburg-Landau equation (GLE) is considered, in which,
the difference method in time and the standard Galerkin FEM are employed.
 Based on the linearized Crank-Nicolson difference method in time and the standard Galerkin finite element method with bilinear element in space, the time-discrete and space-time discrete systems are both constructed.
We focus on a rigorous analysis and consideration of unconditional superconvergence error estimates of the discrete schemes. Firstly, by virtue of the temporal error results, the regularity for the time-discrete system is presented.
Secondly, the classical Ritz projection is used to obtain the spatial error with order $O(h^2)$ in the sense of $L^2-$norm. Thanks to the
relationship between the Ritz projection and the interpolated projection, the superclose estimate with order $O(\tau^2 + h^2)$ in the sense of
$H^1-$norm is derived. Thirdly, it follows from the interpolated postprocessing technique that the global superconvergence result is deduced.
Finally, some numerical results are provided to confirm the theoretical
analysis.
\end{abstract}

\section{Introduction}

In this paper, we are concerned with the numerical solution of the following generalized GLE\\
\begin{equation}\label{11}
 u_t - (\nu+i\eta)\Delta u + (\kappa+i\zeta)f(|u|^2)u - \gamma u = 0,\quad (X, t)\in \Omega\times (0, T],
\end{equation}
with the initial and Dirichlet boundary conditions
\begin{align}
  &u(X, 0)=u_0(X),\quad X\in \Omega, \label{12}\\
  &u(X,t)=0,\quad (X, t)\in \partial \Omega\times (0, T], \label{13}
\end{align}
where $X=(x,y)$, $0<T<+\infty$, $u(X, t)$ and $u_0(X)$  are complex functions, $\nu>0$, $\kappa>0$, $\eta$, $\zeta$, $\gamma$ are given real constants, and
$f(s)$ is a real-value nonlinear function, which is twice continuously differentiable with respect to $s$.
The parameter $\gamma$ is the coefficient of the linear driving term. When $\gamma<0$, all solutions decay to zero.

The NGLE plays an important role in chemistry, engineering, biology and especially in various branches of physics, from nonlinear waves to second-order phase transitions, from superconductivity, superfluidity, and Bose-Einstein condensation to liquid crystals and strings in field theory \cite{aranson2002world,lin2005wave,Ndzana2007Modulated,Kengne20092D,Tsoy2006Dynamical,Du1992Analysis}. Theoretical analysis for the NGLEs has been well done
\cite{Skarka2006Stability,Sotocrespo2012Stability}. However, the analytical expressions of the NGLEs exist only for a few particular cases \cite{Ankiewicz2008Dissipative}. Therefore, there should be significant interest in developing numerical schemes for the approximate solution of the NGLE.
Du \cite{Du1994Finite} studied the semi-discrete and implicit Euler fully-discrete approximations of the NGLEs. In \cite{Wang2011Analysis}, some finite difference schemes with the second-order convergence rate are used to solve the two-dimensional Ginzburg-Landau equation.  Mu and Huang \cite{Mu1998An}
 presented an alternating Crank-Nicolson method for the time-dependent Ginzburg-Landau model of superconductors. In \cite{Xu2011Difference},  three difference schemes of the Ginzburg-Landau Equation in two dimensions were proposed and analyzed, in which, the
nonlinear term was discretized such that the nonlinear iteration was not needed in computation.
However, due to the existence of nonlinearity, error analysis often requires some time-step grid ratio constraints, for example, $\tau=O(h^{d/2}),$ $d=2, 3$ in \cite{Du1994Finite,Mu1998An}, $\tau=O(h)$ in \cite{Xu2011Difference} and $\tau=o(h^2)$ in \cite{Wang2011Analysis}, although numerical tests show the feasibility of the numerical methods for a large time step. To overcome this problem, an error
splitting technique was proposed in \cite{Wang2014A}, and then widely developed in \cite{gao2014optimal,gao2016unconditional,li2013unconditional,li2014unconditionally,si2016unconditional,shi2017unconditional,shi2017unconditional,shi2017unconditional1,wang2014new}.

In this paper, we aim to develop the unconditional superconvergence convergence analysis technique to the generalized
 NGLE. The error function $u^n-U_h^n$ is split into the temporal error $u^n-U^n$ and the spatial error $U^n-U_h^n$.
 By use of different analytical method as \cite{shi2017unconditional}, we also obtain the $H^2-$error estimate of the temporal error with the order $O(\tau^2)$, which plays an important role in the superconvergence analysis. Then, the classical Ritz projection operator $R_h$ is introduced and
 the unconditional error estimate $\|R_hU^n-U_h^n\|_0$ with the order $O(h^2)$ is obtained, which imply that $U_h^n$ is
 unconditionally bounded in the sense of $L^\infty-$norm. Furthermore, we arrive at the superclose property of $\|R_hU^n-U_h^n\|_1$ with the order $O(\tau^2+h^2)$, combining with which and the relation between $R_h$ and the corresponding interpolation operator $I_h$ \cite{Shi2015A}, the error $\|I_hu^n-U_h^n\|_1$ with the order
 $O(\tau^2+h^2)$ holds unconditionally. Besides, the global superconvergence result is obtained by virtue of the interpolated postprocessing technique.
 Finally, some numerical results are displayed to confirm our theoretical analysis.

Throughout the paper, one denotes $(\cdot, \cdot)$ be the $L^2$ inner product function and the corresponding norm is defined by $\|\cdot\|_0:=\|\cdot\|_{L^2(\Omega)}=(\cdot, \cdot)^{1/2}$. Moreover, for any integer $m\geq 0$ and $p\in [1, +\infty]$, let
$W^{m,p}(\Omega)$ be the Sobolev spaces, equipped with the norm
\begin{equation*}
\|\upsilon\|_{m, p}=
  \left\{ \begin{aligned}
        &\bigg(\sum_{|\beta|\leq m}\int_\Omega|\mathcal{D}^{\beta}\upsilon|^pdX\bigg)^{\frac{1}{p}}, &&p\in [1, +\infty), \\
        &\sum_{|\beta|\leq m}ess \sup_{X\in \Omega}|\mathcal{D}^{\beta}\upsilon(X)|, &&p=+\infty,
                          \end{aligned} \right.
\end{equation*}
where
\begin{equation*}
  \mathcal{D}^{\beta}:=\frac{\partial^{|\beta|}}{\partial x_1^{\beta_1}\cdots\partial x_d^{\beta_d}},
\end{equation*}
in which $d$ is the spatial dimension and the multi-index $\beta=(\beta_1, \ldots, \beta_d)$, $\beta_i\geq 0$ and $|\beta|=\sum_{i=1}^d\beta_i$.
When $p=2$, the Sobolev space $W^{m,p}(\Omega)$ is denoted by $H^m(\Omega)$ and the norm $\|\cdot\|_{m, p}$ is simply written as $\|\cdot\|_m$.
Let $H_0^1(\Omega):=\{\upsilon\in H^1(\Omega): v|_{\partial \Omega}=0\}$. Besides, we define the space $L^p(a, b; Y),$ $1\leq p\leq \infty$, equipped with the norm
\begin{equation*}
  \|\upsilon\|_{L^p(a,b; Y)}:=\bigg(\int_a^b\|\upsilon(\cdot, t)\|_Y^pdt\bigg)^{\frac{1}{p}},
\end{equation*}
and
\begin{equation*}
  \|\upsilon\|_{L^\infty(a,b; Y)}:=ess \sup_{t\in [a, b]}\{\|\upsilon(\cdot, t)\|_Y\}.
\end{equation*}

\section{A Linearized Galerkin FEM}

Assume that  $\Omega$ is a rectangle in $(x, y)$ plane with the edges parallel to the two coordinate axes, respectively.
Denote $\mathcal{T}_h$ be a quasiuniform partition of the rectangle $\Omega$.  For any $K\in \mathcal{T}_h$, let $h_K=diam\{K\}$ and $h=\max_{K\in \mathcal{T}_h}\{h_K\}$. Define $\mathcal{V}_h$ be the usual bilinear finite element space, and
$\mathcal{V}_{h0}:=\{v_h\in \mathcal{V}_h: v_h|_{\partial\Omega}=0\}.$ Let $R_h: H_0^1(\Omega)\rightarrow \mathcal{V}_{h0}$ be the associated
Ritz projection operator on $\mathcal{V}_{h0}$, such that
\begin{equation}\label{21}
 (\nabla (u-R_hu), \nabla v_h)=0, \quad \forall v_h\in \mathcal{V}_{h0}.
\end{equation}
From \cite{Thom2006Galerkin}, one obtains that
\begin{equation}\label{22}
  \|\nabla R_hu\|_0\leq C\|\nabla u\|_0,
\end{equation}
and
\begin{equation}\label{23}
  \|u-R_hu\|_0\leq Ch^s\|u\|_s,\quad \forall u\in H_0^1(\Omega)\cap H^2(\Omega),\quad s = 1, 2.
\end{equation}
Meanwhile, it follows from \cite{Shi2015A} that
\begin{equation}\label{24}
  \|I_hu-R_hu\|_1=O(h^2)\|u\|_3,
\end{equation}
where $I_h$ is the associated interpolation operator on $\mathcal{V}_{h0}$.

For any positive integer $N$, we let $\{t_n\}_{n=0}^N$ be a uniform partition of the time interval $[0, T]$ with the step size $\tau=T/N$. Given any sequence of function $\{w^n\}$ defined on $\Omega$, we denote
\begin{equation*}
  \delta_t w^{n}:=\frac{w^{n}-w^{n-1}}{\tau},~\delta_{tt} w^{n}:=\frac{\delta_tw^{n}-\delta_tw^{n-1}}{\tau},~\widetilde{w}^{n}:=\frac{w^{n}+w^{n-1}}{2},~~1\leq n\leq N,
\end{equation*}
and
\begin{equation*}
  \widehat{w}^n:=\frac{3}{2}w^{n-1} - \frac{1}{2}w^{n-2}, \quad 2\leq n\leq N.
\end{equation*}

With these preparations, a linearized Galerkin FEM to the system \eqref{11}-\eqref{13} is obtained, which is to find $U_h^{n}\in \mathcal{V}_{h0}$, such that for
$n\geq 2$,
\begin{equation}\label{25}
  (\delta_t U_h^n, v_h) + (\nu+i\eta)(\nabla \widetilde{U}_h^n, \nabla v_h)
  +(\kappa+i\zeta)(f(|\widehat{U}_h^n|^2)\widetilde{U}_h^n, v_h) - \gamma(\widetilde{U}_h^n, v_h)=0,~\forall v_h\in \mathcal{V}_{h0}.
\end{equation}
This scheme is not selfstarting, and the fist step value $U_h^1\in \mathcal{V}_{h0}$ is given by the following predictor corrector method,
which is to find $U_h^1\in \mathcal{V}_{h0}$, such that
\begin{eqnarray}
  && \bigg(\frac{U_h^{1}-U_h^0}{\tau}, v_h\bigg)+(\nu+i\eta)\bigg(\frac{\nabla U_h^1+\nabla U_h^0}{2}, \nabla v_h\bigg)
  +(\kappa+i\zeta)\bigg(f\bigg(\bigg|\frac{U_h^{1,0}+U_h^0}{2}\bigg|^2\bigg)\frac{U_h^1+U_h^0}{2},\nonumber\\
  &&  v_h\bigg) -\gamma\bigg(\frac{U_h^1+U_h^0}{2}, v_h\bigg),\qquad \forall v_h\in \mathcal{V}_{h0},  \label{26}
\end{eqnarray}
in which, $U_h^{1,0}$ is derived by
\begin{eqnarray}
  && \bigg(\frac{U_h^{1, 0}-U_h^0}{\tau}, v_h\bigg)+(\nu+i\eta)\bigg(\frac{\nabla U_h^{1, 0}+\nabla U_h^0}{2}, \nabla v_h\bigg)
  +(\kappa+i\zeta)\bigg(f(|U_h^0|^2)\frac{U_h^{1, 0}+U_h^0}{2}, v_h\bigg) \nonumber\\
  && -\gamma\bigg(\frac{U_h^{1, 0}+U_h^0}{2}, v_h\bigg),\qquad \forall v_h\in \mathcal{V}_{h0},  \label{27}
\end{eqnarray}
with the initial value $U_h^0:=I_hu_0$. Obviously, for the schemes \eqref{25}-\eqref{27}, only a linear system is needed to be solved at each time step.

\section{Error Estimates for the Time-Discrete System}

In this section, we introduce the following time-discrete system
\begin{equation}
\begin{cases}
\delta_tU^n - (\nu+i\eta)\Delta \widetilde{U}^n+(\kappa+i\zeta)f(|\hat{U}^n|^2)\widetilde{U}^n-\gamma\widetilde{U}^n=0,\quad &(X, t)\in \Omega\times (0, T],
\quad n\geq 2,  \\
U^n=0, &(X, t)\in \partial \Omega\times (0, T],\quad n\geq 1, \\
U(X, 0)=u_0(X), & X\in\Omega.
\end{cases}
\label{31}
\end{equation}
When $n=1$, $U^1$ is determined by
\begin{equation}\label{32}
  \frac{U^{1,0}-U^0}{\tau} - (\nu + i\eta)\frac{\Delta U^{1,0} + \Delta U^0}{2}+
  (\kappa+i\zeta)f(|U^0|^2)\frac{U^{1,0}+U^0}{2}-\gamma\frac{U^{1,0}+U^0}{2}=0,
\end{equation}
and
\begin{equation}\label{33}
  \frac{U^1-U^0}{\tau} - (\nu + i\eta)\frac{\Delta U^1 + \Delta U^0}{2}+
  (\kappa+i\zeta)f\bigg(\bigg|\frac{U^{1,0}+U^0}{2}\bigg|^2\bigg)\frac{U^1+U^0}{2}-\gamma\frac{U^1+U^0}{2}=0,
\end{equation}
where $U^{1,0}|_{\partial \Omega}=0$.

Obviously, the time-discrete system \eqref{31}-\eqref{33}
can be viewed as a system of linear elliptic equations, and thus the existence and
uniqueness of the solution can be obtained immediately \cite{Thom2006Galerkin}.

In what follows, let $e^n:=u^n-U^n, \quad 0\leq n\leq N$. The following theorem gives
the estimate results of $\|e^n\|_0$, $\|e^n\|_1$ and $\|e^n\|_2$, and then the regularity of $U^n$.

\begin{Theo}\label{Theo1} Let $u^n$ and $U^n$ ($0\leq n\leq N$) be the solutions of the system \eqref{11}-\eqref{13} and
the semi-discrete system \eqref{31}, respectively. Suppose $u\in L^2(0,T; H^3(\Omega))$ and $u_t,~u_{tt}\in L^\infty(0, T; H^2(\Omega))$
and
\begin{equation*}
  K_0:=1+\max_{1\leq n\leq N}\|u^n\|_{0, \infty}+\max_{1\leq n\leq N}\|\delta_tu^n\|_{0, \infty}.
\end{equation*}
Then, there exists a positive constant $\tau_0$ such that for $\tau\leq \tau_0$, we have
\begin{equation}\label{34}
  \bigg(\sum_{i=1}^n\tau\|\delta_te^i\|_0^2\bigg)^{\frac{1}{2}}+\|e^n\|_2\leq C_0\tau^2,
\end{equation}
and
\begin{equation}\label{35}
 \|\delta_tU^n\|_{0,\infty}+\|U^n\|_{0,\infty} \leq K_0, \quad \bigg\|\frac{U^{1,0}-U^0}{\tau}\bigg\|_2+\|\delta_{tt}U^n\|_2\leq C_0,
\end{equation}
where $C_0>0$ is a bounded constant.
\end{Theo}
\begin{proof}
From \eqref{11} and \eqref{32}, we obtain the following error equation
\begin{equation}\label{36}
  \frac{e^{1,0}}{\tau}-\frac{\nu+i\eta}{2}\Delta e^{1,0}+\frac{\kappa+i\zeta}{2}f(|u_0|^2)e^{1,0}
  -\frac{\gamma}{2}e^{1,0}=S^{1,0},
\end{equation}
where
\begin{eqnarray*}
  S^{1,0} &:=& \bigg(\frac{u^1-u_0}{\tau}-u_t^{\frac{1}{2}}\bigg)-(\nu+i\eta)\bigg(\frac{\Delta u^1+\Delta u_0}{2}
  -\Delta u^{\frac{1}{2}}\bigg)+(\kappa+i\zeta)\bigg(f(|u_0|^2)\frac{u^1+u_0}{2}-\\
      &&f(|u^{\frac{1}{2}}|^2)u^{\frac{1}{2}}\bigg)-\gamma\bigg(\frac{u^1+u_0}{2}-u^{\frac{1}{2}}\bigg).
\end{eqnarray*}
Obviously, by Taylor's formula, we have
\begin{equation}\label{37}
  \|S^{1,0}\|_0\leq C\tau.
\end{equation}
Multiply \eqref{36} by $e^{1,0}/\tau$ and integrate the resulting equation over $\Omega$ to arrive at
\begin{equation}\label{38}
  \bigg\|\frac{e^{1,0}}{\tau}\bigg\|_0^2+\frac{\nu+i\eta}{2\tau}\|\nabla e^{1,0}\|_0^2
  +\frac{\kappa+i\zeta}{2}\bigg(f(|u_0|^2)e^{1,0}, \frac{e^{1, 0}}{\tau}\bigg)-\frac{\gamma}{2\tau}\|e^{1,0}\|_0^2=\bigg(S^{1,0}, \frac{e^{1,0}}{\tau}\bigg).
\end{equation}
Taking the real part of \eqref{38}, one obtains that
\begin{equation}\label{39}
  \bigg\|\frac{e^{1,0}}{\tau}\bigg\|_0^2+\frac{\nu}{2\tau}\|\nabla e^{1,0}\|_0^2
  +\frac{\kappa}{2}\bigg(f(|u_0|^2)e^{1,0}, \frac{e^{1, 0}}{\tau}\bigg)-\frac{\gamma}{2\tau}\|e^{1,0}\|_0^2=Re\bigg(S^{1,0}, \frac{e^{1,0}}{\tau}\bigg).
\end{equation}
Then, by virtue of Cauchy-Schwarz inequality and Young's inequality, it follows from \eqref{37} and \eqref{39} that
 \begin{equation}\label{310}
  \bigg\|\frac{e^{1,0}}{\tau}\bigg\|_0^2+\frac{\nu}{2\tau}\|\nabla e^{1,0}\|_0^2\leq C\|e^{1,0}\|_0^2
  +\frac{1}{2}\bigg\|\frac{e^{1,0}}{\tau}\bigg\|_0^2+\frac{\gamma\tau}{2}\bigg\|\frac{e^{1,0}}{\tau}\bigg\|_0^2
  +C\tau^2.
 \end{equation}
 Therefore, it exists $C_1>0$ and $\tau_0>0$, such that when $\tau\leq \tau_0$, it holds
 \begin{equation}\label{311}
   \|e^{1,0}\|_0 + \sqrt{\tau}\|\nabla e^{1,0}\|_0\leq C_1\tau^2.
 \end{equation}
In addition, multiplying \eqref{36} by $\Delta e^{1,0}/\tau$, integrating it over $\Omega$, and then taking the imaginary
part of the resulting equation, we easily obtain
\begin{equation*}
  \|\Delta e^{1,0}\|_0\leq C\tau,
\end{equation*}
which implies
\begin{equation}\label{312}
  \|e^{1,0}\|_2\leq C_2\tau.
\end{equation}
Thus, from \eqref{311} and \eqref{312}, we have
 \begin{equation}\label{313}
   \|U^{1,0}\|_{0, \infty}\leq \|e^{1,0}\|_{0, \infty}+\|u^{1}\|_{0, \infty}\leq C\|e^{1,0}\|_2 + \|u^{1}\|_{0, \infty}\leq CC_2\tau + \|u^1\|_{0, \infty}\leq K_0,
 \end{equation}
 and
 \begin{equation}\label{314}
   \bigg\|\frac{U^{1,0}-U^0}{\tau}\bigg\|_2\leq \bigg\|\frac{e^{1,0}}{\tau}\bigg\|_2
   +\bigg\|\frac{u^{1}-u_0}{\tau}\bigg\|_2\leq C_3,
 \end{equation}
 where $\tau\leq \tau_1:=1/CC_2$.

Next, from \eqref{11} and \eqref{33}, the error equation at $t=t_1$ is obtained
  \begin{equation}\label{315}
  \frac{e^{1}}{\tau}-\frac{\nu+i\eta}{2}\Delta e^{1}+(\kappa+i\zeta)P^1_1
  -\frac{\gamma}{2}e^{1}=S^{1},
\end{equation}
where
\begin{equation*}
  P_1^1:= f\bigg(\bigg|\frac{u^1+u_0}{2}\bigg|^2\bigg)\frac{u^1+u_0}{2}-f\bigg(\bigg|\frac{U^{1,0}+U^0}{2}\bigg|^2\bigg)\frac{U^1+U^0}{2},
\end{equation*}
and
\begin{eqnarray*}
  S^{1} &:=& \bigg(\frac{u^1-u_0}{\tau}-u_t^{\frac{1}{2}}\bigg)-(\nu+i\eta)\bigg(\frac{\Delta u^1+\Delta u_0}{2}
  -\Delta u^{\frac{1}{2}}\bigg)+(\kappa+i\zeta)\bigg(f\bigg(\bigg|\frac{u^1+u_0}{2}\bigg|^2\bigg)\\
      &&\frac{u^1+u_0}{2}-f(|u^{\frac{1}{2}}|^2)u^{\frac{1}{2}}\bigg)-\gamma\bigg(\frac{u^1+u_0}{2}-u^{\frac{1}{2}}\bigg).
\end{eqnarray*}
Obviously, by Taylor's formula, one obtains that
\begin{equation}\label{316}
  \|S^1\|_0\leq C\tau^2.
\end{equation}
We multiply \eqref{315} by $e^{1}/\tau$ and integrate the resulting equation over $\Omega$ to arrive at
\begin{equation}\label{317}
  \bigg\|\frac{e^{1}}{\tau}\bigg\|_0^2+\frac{\nu+i\eta}{2\tau}\|\nabla e^1\|_0^2 + (\kappa+i\zeta)\bigg(P_1^1, \frac{e^1}{\tau}\bigg)
  -\frac{\gamma}{2\tau}\|e^1\|_0^2=\bigg(S^1, \frac{e^1}{\tau}\bigg).
\end{equation}
Take the real part of \eqref{317} to obtain
\begin{equation}\label{318}
  \bigg\|\frac{e^{1}}{\tau}\bigg\|_0^2+\frac{\nu}{2\tau}\|\nabla e^1\|_0^2 + Re\bigg\{(\kappa+i\zeta)\bigg(P_1^1, \frac{e^1}{\tau}\bigg)\bigg\}
  -\frac{\gamma}{2\tau}\|e^1\|_0^2=Re\bigg(S^1, \frac{e^1}{\tau}\bigg).
\end{equation}
Thanks to \eqref{313}, we have
\begin{eqnarray*}
  |P_1^1| &=& \bigg |f\bigg(\bigg|\frac{u^1+u_0}{2}\bigg|^2\bigg)\frac{u^1+u_0}{2}-f\bigg(\bigg|\frac{U^{1,0}+U^0}{2}\bigg|^2\bigg)\frac{U^1+U^0}{2}\bigg|\\
    & = & \bigg |\bigg[f\bigg(\bigg|\frac{u^1+u_0}{2}\bigg|^2\bigg)-f\bigg(\bigg|\frac{U^{1,0}+U^0}{2}\bigg|^2\bigg)\bigg]\frac{u^1+u_0}{2}
    + f\bigg(\bigg|\frac{U^{1,0}+U^0}{2}\bigg|^2\bigg)\frac{e^1}{2}\bigg|   \\
    &\leq & C(|e^{1,0}|+|e^1|).
\end{eqnarray*}
Therefore, it follows from \eqref{311} that
\begin{equation}\label{319}
  \bigg|Re\bigg\{(\kappa+i\zeta)\bigg(P_1^1, \frac{e^1}{\tau}\bigg)\bigg\}\bigg|\leq C(\|e^{1,0}\|_0^2+\|e^1\|_0^2)+\frac{1}{4}\bigg\|\frac{e^{1}}{\tau}\bigg\|_0^2\leq C\|e^1\|_0^2+C\tau^4+\frac{1}{4}\bigg\|\frac{e^{1}}{\tau}\bigg\|_0^2.
\end{equation}
By virtue of Cauchy-Schwarz inequality and Young's inequality, one obtains from \eqref{316} that
\begin{equation}\label{320}
  \bigg|Re\bigg(S^1, \frac{e^1}{\tau}\bigg)\bigg|\leq C\|S^1\|_0^2 + \frac{1}{4}\bigg\|\frac{e^{1}}{\tau}\bigg\|_0^2\leq C\tau^4 + \frac{1}{4}\bigg\|\frac{e^{1}}{\tau}\bigg\|_0^2.
\end{equation}
Substituting \eqref{319}-\eqref{320} into \eqref{318} reads
\begin{equation}\label{321}
  \frac{1}{2}\bigg\|\frac{e^{1}}{\tau}\bigg\|_0^2+\frac{\nu}{2\tau}\|\nabla e^1\|_0^2\leq
  \frac{\gamma}{2\tau}\|e^1\|_0^2+C\|e^1\|_0^2+C\tau^4.
\end{equation}
Then, it is apparent to see that there exist $C_4>0$ and $\tau_2>0$, such that when $\tau\leq \tau_2$, we have
\begin{equation}\label{322}
 \|e^1\|_0 + \sqrt{\tau}\|\nabla e^1\|_0\leq C_4\tau^3.
\end{equation}
Next, we take the inner product of \eqref{315} by $\Delta e^1$ to arrive at
\begin{equation}\label{323}
  -\|\nabla e^1\|_0^2-\frac{(\nu+i\eta)\tau}{2}\|\Delta e^1\|_0^2
  +(\kappa + i\zeta)\tau(P_1^1, \Delta e^1)+\frac{\gamma\tau}{2}\|\nabla e^1\|^2_0
  = \tau(S^1, \Delta e^1).
\end{equation}
Take the imaginary part of \eqref{323} to obtain
\begin{equation}\label{324}
  -\frac{\eta}{2}\|\Delta e^1\|_0^2+Im\bigg\{(\kappa+i\zeta)(P_1^1, \Delta e^1)\bigg\}
  =Im(S^1, \Delta e^1).
\end{equation}
By using \eqref{316} and $|P_1^1|\leq C(|e^{1,0}|+|e^1|)$, we easily conclude
\begin{equation}\label{325}
  \|\Delta e^1\|_0\leq C\tau^2,
\end{equation}
which further implies that
\begin{equation}\label{326}
  \|e^1\|_2\leq C_5\tau^2.
\end{equation}
Therefore, it follows from \eqref{326} that
\begin{equation}\label{327}
  \|\delta_{tt}U^1\|_2\leq \|\delta_{tt}e^1\|_2+\|\delta_{tt}u^1\|_2\leq C_6,
\end{equation}
and
\begin{eqnarray}\label{328}
  \|\delta_tU^1\|_{0,\infty}+\|U^1\|_{0,\infty} & \leq& \|\delta_te^1\|_{0,\infty}+\|\delta_tu^1\|_{0,\infty}+\|e^1\|_{0,\infty}+\|u^1\|_{0,\infty} \nonumber \\
   &\leq& C(\|\delta_te^1\|_{2}+\|e^1\|_{2})+\|\delta_tu^1\|_{0,\infty}+\|u^1\|_{0,\infty} \nonumber \\
   &\leq & CC_5(\tau+\tau^2) + \|\delta_tu^1\|_{0,\infty}+\|u^1\|_{0,\infty} \nonumber \\
   & \leq & K_0,
\end{eqnarray}
where $\tau\leq \tau_3$ and $CC_5(\tau_3+\tau_3^2)\leq 1$.

By virtue of the mathematical induction method, one assumes that \eqref{34} holds for $m\leq n-1$. Then, under this assumption, we have for $m\leq n-1$
\begin{eqnarray}\label{329}
  \|\delta_tU^m\|_{0,\infty}+\|U^m\|_{0,\infty} & \leq& \|\delta_te^m\|_{0,\infty}+\|\delta_tu^m\|_{0,\infty}+\|e^m\|_{0,\infty}+\|u^m\|_{0,\infty} \nonumber \\
   &\leq& C(\|\delta_te^m\|_{2}+\|e^m\|_{2})+\|\delta_tu^m\|_{0,\infty}+\|u^m\|_{0,\infty} \nonumber \\
   &\leq & CC_0(2\tau + \tau^2) + \|\delta_tu^m\|_{0,\infty}+\|u^m\|_{0,\infty} \nonumber \\
   & \leq & K_0,
\end{eqnarray}
where $\tau\leq \tau_4$ and $CC_0(2\tau_4+\tau_4^2)\leq 1$. Now, we intend to prove that  \eqref{34} also holds for $m=n$.
From \eqref{11} and \eqref{31}, the error equation at $t=t_n$ is obtained
\begin{equation}\label{330}
  \delta_te^n-(\nu+i\eta)\Delta\widetilde{e}^n+(\kappa+i\zeta)P_1^n
  -\gamma\widetilde{e}^n=S^n,
\end{equation}
where
\begin{equation*}
  P_1^n:= f(|\widehat{u}^n|^2)\widetilde{u}^n - f(|\widehat{U}^n|^2)\widetilde{U}^n,
\end{equation*}
and
\begin{equation*}
  S^n:= \big(\delta_tu^n-u_t^{n-\frac{1}{2}}\big)-(\nu+i\eta)\Delta\big(\widetilde{u}^n
  -u^{n-\frac{1}{2}}\big)+(\kappa+i\zeta)\bigg(f(|\widehat{u}^n|^2)\widetilde{u}^n - f(|u^{n-\frac{1}{2}}|^2)u^{n-\frac{1}{2}}\bigg).
\end{equation*}
Obviously, we have
\begin{equation}\label{331}
  \|S^n\|_0\leq C\tau^2.
\end{equation}

Multiply \eqref{330} by $\widetilde{e}^n$, and integrate the resulting equation over $\Omega$ to arrive at
\begin{equation}\label{332}
  (\delta_te^n, \widetilde{e}^n)+(\nu+i\eta)\|\nabla\widetilde{e}^n\|_0^2+(\kappa+i\zeta)(P_1^n, \widetilde{e}^n)
  -\gamma\|\widetilde{e}^n\|_0^2=(S^n, \widetilde{e}^n).
\end{equation}
Taking the real part of \eqref{332}, and thanks to
\begin{equation}\label{333}
  Re(\delta_te^n, \widetilde{e}^n)=\frac{\|e^n\|_0^2-\|e^{n-1}\|_0^2}{2\tau},
\end{equation}
it follows that
\begin{equation}\label{334}
  \frac{\|e^n\|_0^2-\|e^{n-1}\|_0^2}{2\tau}+\nu\|\nabla\widetilde{e}^n\|_0^2=\gamma\|\widetilde{e}^n\|_0^2
  -Re\big\{(\kappa+i\zeta)(P_1^n, \widetilde{e}^n)\big\}+Re(S^n, \widetilde{e}^n).
\end{equation}
Then, one obtains from \eqref{329} that
\begin{equation}\label{335}
\|P_1^n\|_0\leq \big\|(f(|\widehat{u}^n|^2)-f(|\widehat{U}^n|^2))\widetilde{u}^n\big\|
   +\big\|f(|\widehat{U}^n|^2)\widetilde{e}^n\big\|\leq C\big(\|e^{n}\|_0+\|e^{n-1}\|_0+\|e^{n-2}\|_0\big).
\end{equation}
From \eqref{331}, \eqref{334} and \eqref{335}, we have
\begin{equation}\label{336}
  \|e^n\|_0^2\leq \|e^1\|_0^2+C\tau\sum_{i=0}^n\|e^i\|_0^2+C\tau^4.
\end{equation}
By applying \eqref{322} and the discrete Gronwall's inequality, we derive that
\begin{equation}\label{337}
  \|e^n\|_0\leq C_6\tau^2.
\end{equation}

Next, we multiply \eqref{330} by $\delta_t{e}^n$, and integrate the resulting equation over $\Omega$ to get
\begin{equation}\label{338}
  \frac{1}{\nu+i\eta}\|\delta_te^n\|_0^2-(\Delta \widetilde{e}^n, \delta_t e^n)
  +\frac{\kappa+i\zeta}{\nu+i\eta}(P_1^n, \delta_t e^n)-\frac{\gamma}{\nu+i\eta}(\widetilde{e}^n, \delta_t e^n)
  =\frac{1}{\nu+i\eta}(S^n, \delta_t e^n).
\end{equation}
Take the real part of \eqref{338} to arrive at
\begin{eqnarray}
  &&\frac{\nu}{\nu^2+\eta^2}\|\delta_te^n\|_0^2+\frac{\|\nabla e^n\|^2-\|\nabla e^{n-1}\|^2}{2\tau} \nonumber\\
  && = -Re\bigg\{\frac{\kappa+i\zeta}{\nu+i\eta}(P_1^n, \delta_t e^n)\bigg\}+Re\bigg\{\frac{\gamma}{\nu+i\eta}(\widetilde{e}^n, \delta_t e^n)\bigg\}+Re\bigg\{\frac{1}{\nu+i\eta}(S^n, \delta_t e^n)\bigg\}.\nonumber\\
  && \label{339}
\end{eqnarray}
We now analyze the three terms at the right hand of \eqref{339}. For the first term, it follows that
\begin{equation}\label{340}
  \bigg|Re\bigg\{\frac{\kappa+i\zeta}{\nu+i\eta}(P_1^n, \delta_t e^n)\bigg\}\bigg|\leq
  \frac{\sqrt{\kappa^2+\zeta^2}}{\sqrt{\nu^2+\eta^2}}\bigg(\frac{1}{4\epsilon_1}\|P_1^n\|^2_0+\epsilon_1\|\delta_t e^n\|^2_0\bigg).
\end{equation}
Then, for the second and third terms, we have
\begin{equation}\label{341}
  \bigg|Re\bigg\{\frac{\gamma}{\nu+i\eta}(\widetilde{e}^n, \delta_t e^n)\bigg\}\bigg\}\bigg|\leq
  \frac{|\gamma|}{\sqrt{\nu^2+\eta^2}}\bigg(\frac{1}{4\epsilon_2}\|\widetilde{e}^n\|^2_0+\epsilon_2\|\delta_t e^n\|^2_0\bigg),
\end{equation}
and
\begin{equation}\label{342}
  \bigg|Re\bigg\{\frac{1}{\nu+i\eta}(R^n, \delta_t e^n)\bigg\}\bigg|\leq
  \frac{1}{\sqrt{\nu^2+\eta^2}}\bigg(\frac{1}{4\epsilon_3}\|S^n\|^2_0+\epsilon_3\|\delta_t e^n\|^2_0\bigg).
\end{equation}
Take appropriate $\epsilon_1$, $\epsilon_2$ and $\epsilon_3$, such that
\begin{equation}\label{343}
  \frac{\sqrt{\kappa^2+\zeta^2}}{\sqrt{\nu^2+\eta^2}}\epsilon_1+ \frac{|\gamma|}{\sqrt{\nu^2+\eta^2}}\epsilon_2+\frac{1}{\sqrt{\nu^2+\eta^2}}\epsilon_3
  =\frac{\nu}{2(\nu^2+\eta^2)}.
\end{equation}
Then, it follows from \eqref{331}, \eqref{335}, \eqref{337}, \eqref{339}-\eqref{342} that
\begin{equation}\label{344}
  \|\delta_te^n\|_0^2+\frac{\|\nabla e^n\|^2-\|\nabla e^{n-1}\|^2}{\tau}\leq C(\|\nabla e^n\|_0^2+\|\nabla e^{n-1}\|_0^2+\|\nabla e^{n-2}\|_0^2)
  +C\tau^4.
\end{equation}
Summing \eqref{344} from $2$ to $n$ yields that
\begin{equation}\label{345}
  \sum_{i=2}^n\tau\|\delta_te^i\|_0^2+\|\nabla e^n\|_0^2\leq C\tau \sum_{i=0}^n\|\nabla e^i\|_0^2 + C\tau^4
  \leq C\tau \sum_{i=2}^n\bigg(\sum_{j=2}^i\tau\|\delta_te^j\|_0^2+\|\nabla e^i\|_0^2\bigg)+C\tau^4,
\end{equation}
which reduces to (via discrete Gronwall's inequality)
\begin{equation}\label{346}
  \bigg(\sum_{i=2}^n\tau\|\delta_te^i\|_0^2\bigg)^{\frac{1}{2}}+\|\nabla e^n\|_0\leq C_7\tau^2.
\end{equation}
From \eqref{322}, we also have
\begin{equation}\label{eq347}
  \bigg(\sum_{i=1}^n\tau\|\delta_te^i\|_0^2\bigg)^{\frac{1}{2}}+\|\nabla e^n\|_0\leq C_7\tau^2.
\end{equation}

In what follows, we multiply \eqref{330} by $-\delta_t \Delta e^n$, and integrate it over $\Omega$ to obtain that
%
\begin{eqnarray}\label{347}
  &&  \frac{1}{\nu+i\eta}\|\delta_t\nabla e^n\|_0^2+(\Delta \widetilde{e}^n, \delta_t \Delta e^n) \nonumber\\
   && = \frac{\kappa+i\zeta}{\nu+i\eta}(P_1^n, \delta_t \Delta e^n)+\frac{\gamma}{\nu+i\eta}(\nabla\widetilde{e}^n, \delta_t \nabla e^n)
   - \frac{1}{\nu+i\eta}(S^n, \delta_t \Delta e^n).
\end{eqnarray}
Take the real part of \eqref{347} to arrive at
\begin{eqnarray}\label{348}
  &&  \frac{\nu}{\nu^2+\eta^2}\|\delta_t\nabla e^n\|_0^2+\frac{\|\Delta e^n\|_0^2-\|\Delta e^{n-1}\|_0^2}{2\tau} \nonumber\\
   && = Re\bigg\{\frac{\kappa+i\zeta}{\nu+i\eta}(P_1^n, \delta_t \Delta e^n)\bigg\}+Re\bigg\{\frac{\gamma}{\nu+i\eta}(\nabla\widetilde{e}^n, \delta_t \nabla e^n)\bigg\}
   - Re\bigg\{\frac{1}{\nu+i\eta}(S^n, \delta_t \Delta e^n)\bigg\}.\nonumber\\
   &&
\end{eqnarray}
For the second term at the right hand of \eqref{348}, by virtue of the Cauchy-Schwarz inequality and Young's inequality, one easily obtains that
\begin{equation}\label{349}
  \bigg|Re\bigg\{\frac{\gamma}{\nu+i\eta}(\nabla\widetilde{e}^n, \delta_t \nabla e^n)\bigg\}\bigg|\leq
  \frac{|\gamma|}{\sqrt{\nu^2+\eta^2}}\bigg(\frac{1}{4\epsilon_4}\|\nabla\widetilde{e}^n\|^2_0+\epsilon_4\|\delta_t \nabla e^n\|^2_0\bigg).
\end{equation}
As to the first and third terms at the right hand of \eqref{348}, one needs to transfer $\tau$ from one part of the inner product to the other. To this end, we assume $\widehat{u}^1=\widetilde{u}^1$,
 $\widehat{U}^1=\widetilde{U}^1$ and $\widehat{e}^1=\widetilde{e}^1$. Rewrite $(\Lambda^n, \delta_t \Delta e^n)$ $(\Lambda^n = P_1^n$ or $S^n)$ as
    \begin{equation}\label{350}
      (\Lambda^n, \delta_t \Delta e^n) = - (\delta_t\Lambda^n, \Delta e^{n-1}) + \delta_t (\Lambda^n, \Delta e^n).
    \end{equation}
As the result of the mathematical induction \eqref{329}, we easily deduce that
\begin{eqnarray}\label{351}
 \|\delta_tP_1^n\|_0 &\leq& C\big(\|\delta_t\widetilde{e}^n\|_0 + \|\widehat{e}^{n-1}\|_0+\|\widetilde{e}^n\|_0\big)+\nonumber\\
  && C\bigg\|\frac{[f(|\widehat{u}^n|^2)-f(|\widehat{u}^{n-1}|^2)]-
   [f(|\widehat{U}^n|^2)-f(|\widehat{U}^{n-1}|^2)]}{\tau}\bigg\|_0.
\end{eqnarray}
Similar as \cite{shi2017unconditional}, we have
\begin{eqnarray}\label{352}
    &&\bigg\|\frac{[f(|\widehat{u}^n|^2)-f(|\widehat{u}^{n-1}|^2)]-
   [f(|\widehat{U}^n|^2)-f(|\widehat{U}^{n-1}|^2)]}{\tau}\bigg\|_0\nonumber\\
  &&\leq C\big(\|\delta_t\widehat{e}^n\|_0
   +\|\widehat{e}^{n-1}\|_0+\|\widehat{e}^n\|_0\big)+C\tau^2.
\end{eqnarray}
Then, it follows that
\begin{equation}\label{353}
  \|\delta_tP_1^n\|_0 \leq C\big(\|\delta_t\widehat{e}^n\|_0
   +\|\widehat{e}^{n-1}\|_0+\|\widehat{e}^n\|_0+\|\delta_t\widetilde{e}^n\|_0 +\|\widetilde{e}^n\|_0\big)+C\tau^2.
\end{equation}
Accordingly, from \eqref{350}-\eqref{353}, we obtain that
\begin{eqnarray}\label{354}
  Re\bigg\{\frac{\kappa+i\zeta}{\nu+i\eta}(P_1^n, \delta_t \Delta e^n)\bigg\}&=&-Re\bigg\{\frac{\kappa+i\zeta}{\nu+i\eta}(\delta_tP_1^n, \Delta e^{n-1})\bigg\}
  +Re\bigg\{\frac{\kappa+i\zeta}{\nu+i\eta}\delta_t(P_1^n, \Delta e^{n-1})\bigg\}  \nonumber\\
  &\leq&   C\big(\|\delta_t\widehat{e}^n\|^2_0
   +\|\widehat{e}^{n-1}\|^2_0+\|\widehat{e}^n\|^2_0+\|\delta_t\widetilde{e}^n\|^2_0 +\|\widetilde{e}^n\|^2_0+\|\Delta e^{n-1}\|_0^2\big)\nonumber\\
  && + C\tau^4  +Re\bigg\{\frac{\kappa+i\zeta}{\nu+i\eta}\delta_t(P_1^n, \Delta e^{n-1})\bigg\}.
\end{eqnarray}
On the other hand, it is easy to show that
\begin{eqnarray}\label{355}
  \|\delta_tS^n\|_0 &\leq& \bigg\|\frac{\big(\delta_tu^n-u_t^{n-\frac{1}{2}}\big)-\big(\delta_tu^{n-1}-u_t^{n-\frac{3}{2}}\big)}{\tau}\bigg\|_0
  \nonumber\\
  && + \sqrt{\nu^2+\eta^2}\bigg\|\frac{\big(\Delta\widetilde{u}^n-\Delta u^{n-\frac{1}{2}}\big)-\big(\Delta\widetilde{u}^{n-1}-\Delta u^{n-\frac{3}{2}}\big)}{\tau}\bigg\|_0   \nonumber\\
   &&  + \bigg\|\frac{[f(|\widehat{u}^n|^2)\widetilde{u}^n - f(|u^{n-\frac{1}{2}}|^2)u^{n-\frac{1}{2}}]-
   [f(|\widehat{u}^{n-1}|^2)\widetilde{u}^{n-1} - f(|u^{n-\frac{3}{2}}|^2)u^{n-\frac{3}{2}}]}{\tau}\bigg\|_0\nonumber\\
   &\leq&  C\tau^2.
\end{eqnarray}
Therefore, it follows from \eqref{350} and \eqref{355} that
\begin{eqnarray}
   - Re\bigg\{\frac{1}{\nu+i\eta}(S^n, \delta_t \Delta e^n)\bigg\} &=& Re\bigg\{\frac{1}{\nu+i\eta}(\delta_tS^n, \Delta e^{n-1})\bigg\} - Re\bigg\{\frac{1}{\nu+i\eta}\delta_t (S^n, \Delta e^n)\bigg\} \nonumber\\
   & \leq& C\tau^4 + C\|\Delta e^{n-1}\|_0^2 - Re\bigg\{\frac{1}{\nu+i\eta}\delta_t (S^n, \Delta e^n)\bigg\}.\label{356}
\end{eqnarray}
Substituting \eqref{349}, \eqref{354} and \eqref{356} into \eqref{348}, and taking
\begin{equation*}
  \epsilon_4=\frac{\nu}{2|\gamma|\sqrt{\nu^2+\eta^2}},
\end{equation*}
we arrive at one by summing from $2$ to $n$,
\begin{eqnarray}\label{357}
 \|\Delta e^n\|_0^2 &\leq& C\tau^4+C\tau\sum_{k=1}^n\bigg(\|\delta_te^k\|_0^2+\|\Delta e^k\|_0^2\bigg)+\|\Delta e^1\|_0^2
 +C\|\delta_t\widehat{e}^2\|_0^2+C\|\Delta \widehat{e}^1\|_0^2 +
  \nonumber\\
  &&  Re\bigg\{\frac{\kappa+i\zeta}{\nu+i\eta}\big[(P_1^n, \Delta e^n)-(P_1^1, \Delta e^1)\big]\bigg\}- Re\bigg\{\frac{1}{\nu+i\eta}\big[(S^n, \Delta e^n)-(S^1, \Delta e^1)\big]\bigg\}.\nonumber\\
\end{eqnarray}
Then, thanks to \eqref{335} and $e^n=\tau\sum_{i=1}^n\delta_te^i$,
we derive that
\begin{equation}\label{358}
\|P_1^n\|_0^2\leq C\big(\|e^{n}\|^2_0+\|e^{n-1}\|^2_0+\|e^{n-2}\|^2_0\big)\leq C\tau^2 \bigg\|\sum_{k=1}^n\delta_te^k\bigg\|^2_0
\leq C\tau\sum_{k=1}^n\|\delta_te^k\|_0^2.
\end{equation}
Therefore, it follows from \eqref{357} that
\begin{equation}\label{359}
  \|\Delta e^n\|_0^2\leq C\tau^4+ C\tau\sum_{k=1}^n\bigg(\|\delta_te^k\|_0^2+\|\Delta e^k\|_0^2\bigg).
\end{equation}
Hence, from \eqref{346} and \eqref{359}, and then applying the discrete Gronwall's inequality, one obtains that
\begin{equation}\label{360}
  \|\Delta e^n\|_0\leq C\tau^2,
\end{equation}
which implies that
\begin{equation}\label{361}
  \|e^n\|_2\leq C_8\tau^2.
\end{equation}
Finally, from \eqref{361}, there hold
\begin{equation}\label{362}
 \|\delta_{tt}U^n\|_2\leq \|\delta_{tt}e^n\|_2+\|\delta_{tt}u^n\|_2\leq C_9,
\end{equation}
and
\begin{eqnarray}\label{363}
  \|\delta_tU^n\|_{0,\infty}+\|U^n\|_{0,\infty} & \leq& \|\delta_te^n\|_{0,\infty}+\|\delta_tu^n\|_{0,\infty}+\|e^n\|_{0,\infty}+\|u^n\|_{0,\infty} \nonumber \\
   &\leq & C\bigg(\bigg\|\frac{e^n-e^{n-1}}{\tau}\bigg\|_2+\|e^n\|_{2})\bigg)+\|\delta_tu^n\|_{0,\infty}+\|u^n\|_{0,\infty} \nonumber \\
   &\leq & CC_8(2\tau+\tau^2) + \|\delta_tu^n\|_{0,\infty}+\|u^n\|_{0,\infty} \nonumber \\
   & \leq & K_0,
\end{eqnarray}
where $\tau\leq \tau_5$ and $\tau_5$ satisfies $CC_8(2\tau_5+\tau_5^2)\leq 1$. Suppose that $C_0\leq \sum_{k=0}^9C_k$ and $\tau\leq \min_{0\leq k\leq 5}\{\tau_k\}$. Then, we have now finished the mathematical induction, and thus complete the proof of Theorem \ref{Theo1}. $\Box$

\end{proof}

\section{Superconvergence Results for the Fully Discrete System}

In this section, we intent to estimate the global superconvergence result for the fully discrete system. To this end,
we first show the unconditional boundedness of the fully discrete solution $U_h^n$ ($1\leq n\leq N$) in the sense of $L^\infty-$norm,
which can be deduced by the error result $\|R_hU^n-U_h^n\|_0=O(h^2)$. Then, the error $\|\nabla(R_hU^n-U_h^n)\|_0$ is bounded with the order
$O(h^2+\tau^2)$. According to the relationship between $I_h$ and $R_h$,
the superclose result $\|I_hu^n-U_h^n\|_0=O(h^2+\tau^2)$ is proved. Finally, the global superconvergence result is derived by virtue of
the interpolated postprocessing technique. For convenience, we split the error functions into following ones
\begin{equation}\label{41}
  U^n-U_h^n=(U^n-R_hU^n)+(R_hU^n-U_h^n):=\lambda^n+\theta^n, \quad 0\leq n\leq N.
\end{equation}

\begin{Theo}\label{Theo2}
Let $u$ and $U_h^n$ ($1\leq n\leq N$) be the solutions of \eqref{11}-\eqref{13} and \eqref{25}-\eqref{27}, respectively.
Then, under the conditions of Theorem \ref{Theo1}, we obtain that
\begin{equation}\label{42}
  \|\theta^n\|_0\leq C'h^2, \quad 0\leq n\leq N.
\end{equation}
Furthermore, if $h$ is small enough, we have
 \begin{equation}\label{43}
  \|U_h^n\|_{0, \infty}\leq K_1, \quad 0\leq n\leq N,
\end{equation}
where $C'>0$ is a constant independent of $h$, and
\begin{equation*}
  K_1:=1+\|R_hU^{1,0}\|_{0, \infty}+\max_{0\leq j\leq N}\|R_hU^{j}\|_{0, \infty}.
\end{equation*}
\end{Theo}
\begin{proof}
By the definitions of $U_h^0$ and $K_1$, it obviously holds that
\begin{equation}\label{44}
  \|U_h^{0}\|_{0, \infty}=\|R_hU^{0}\|_{0, \infty}\leq K_1.
\end{equation}
From \eqref{27} and \eqref{32}, we derive the error equation
\begin{eqnarray}\label{45}
    && \bigg(\frac{\theta^{1,0}}{\tau}, v_h\bigg)+\frac{\nu+i\eta}{2}(\nabla\theta^{1,0}, \nabla v_h)
    +(\kappa+i\zeta)\bigg(f(|U^0|^2)\frac{U^{1,0}+U^0}{2}-f(|U_h^0|^2)\frac{U_h^{1,0}+U_h^0}{2}, v_h\bigg) \nonumber \\
    && - \frac{\gamma}{2}(\theta^{1,0}, v_h)= -\bigg(\frac{\lambda^{1,0}-\lambda^0}{\tau}, v_h\bigg)
    -\frac{\nu+i\eta}{2}\big(\nabla(\lambda^{1,0}+\lambda^0), \nabla v_h\big)+\frac{\gamma}{2}(\lambda^{1,0}+\lambda^0, v_h). \nonumber \\
    &&
\end{eqnarray}
Substituting $v_h={\theta^{1,0}}/{\tau}$ in \eqref{45}, and taking the real part of the resulting equation, one obtains that
\begin{eqnarray}\label{46}
    && \bigg\|\frac{\theta^{1,0}}{\tau}\bigg\|_0^2+\frac{\nu}{2\tau}\|\nabla\theta^{1,0}\|_0^2- \frac{\gamma\tau}{2}\bigg\|\frac{\theta^{1,0}}{\tau}\bigg\|_0^2
    \nonumber \\
    && =-Re\bigg\{(\kappa+i\zeta)\bigg(f(|U^0|^2)\frac{U^{1,0}+U^0}{2}-f(|U_h^0|^2)\frac{U_h^{1,0}+U_h^0}{2}, \frac{\theta^{1,0}}{\tau}\bigg) \bigg\}-Re\bigg(\frac{\lambda^{1,0}-\lambda^0}{\tau}, \frac{\theta^{1,0}}{\tau}\bigg)
     \nonumber \\
    &&~~~~+\frac{\gamma}{2}Re\bigg(\lambda^{1,0} +\lambda^0, \frac{\theta^{1,0}}{\tau}\bigg):=\sum_{j=1}^3\mathcal{L}_j.
\end{eqnarray}
It follows from \eqref{23} and \eqref{313} that
\begin{eqnarray}\label{47}
  |\mathcal{L}_1| &=& \bigg|Re\bigg\{(\kappa+i\zeta)\bigg(f(|U^0|^2)\frac{U^{1,0}+U^0}{2}-f(|U_h^0|^2)\frac{U_h^{1,0}+U_h^0}{2}, \frac{\theta^{1,0}}{\tau}\bigg) \bigg\}\bigg| \nonumber\\
  &=& \bigg|Re\bigg\{(\kappa+i\zeta)\bigg((f(|U^0|^2)-f(|U_h^0|^2))\frac{U^{1,0}+U^0}{2}+f(|U_h^0|^2)\frac{\theta^{1,0}+\lambda^{1,0}+\lambda^0}{2}, \frac{\theta^{1,0}}{\tau}\bigg) \bigg\}\bigg|   \nonumber\\
   &\leq & Ch^4+C\|\theta^{1,0}\|_0^2+\epsilon_5\bigg\|\frac{\theta^{1,0}}{\tau}\bigg\|_0^2.
\end{eqnarray}
Similarly, from \eqref{23} and \eqref{35}, we obtain
\begin{equation}\label{48}
   |\mathcal{L}_2| =\bigg|Re\bigg(\frac{\lambda^{1,0}-\lambda^0}{\tau}, \frac{\theta^{1,0}}{\tau}\bigg)\bigg|
  \leq Ch^2\bigg\|\frac{U^{1,0}-U^0}{\tau}\bigg\|_2\bigg\| \frac{\theta^{1,0}}{\tau}\bigg\|_0\leq Ch^4
  +\epsilon_6\bigg\|\frac{\theta^{1,0}}{\tau}\bigg\|_0^2,
\end{equation}
and
\begin{equation}\label{49}
   |\mathcal{L}_3| =\bigg|\frac{\gamma}{2}Re\bigg(\lambda^{1,0} +\lambda^0, \frac{\theta^{1,0}}{\tau}\bigg)\bigg|
  \leq Ch^2\bigg\|U^{1,0}+U^0\bigg\|_2\bigg\| \frac{\theta^{1,0}}{\tau}\bigg\|_0\leq Ch^4
  +\epsilon_7\bigg\|\frac{\theta^{1,0}}{\tau}\bigg\|_0^2.
\end{equation}
Substituting \eqref{47}-\eqref{49} into \eqref{46} yields
\begin{equation}\label{410}
   \bigg\|\frac{\theta^{1,0}}{\tau}\bigg\|_0^2+\frac{\nu}{2\tau}\|\nabla\theta^{1,0}\|_0^2- \frac{\gamma\tau}{2}\bigg\|\frac{\theta^{1,0}}{\tau}\bigg\|_0^2
   \leq Ch^4+C\|\theta^{1,0}\|_0^2+(\epsilon_5+\epsilon_6+\epsilon_7)\bigg\|\frac{\theta^{1,0}}{\tau}\bigg\|_0^2.
\end{equation}
Taking $\epsilon_5+\epsilon_6+\epsilon_7=1/2$ in \eqref{410}, then we easily conclude that
\begin{equation}\label{411}
 \|\theta^{1,0}\|_0+\sqrt{\tau}\|\nabla \theta^{1,0}\|_0\leq C'_0\tau h^2,
\end{equation}
which implies
\begin{eqnarray}\label{412}
  \|U_h^{1,0}\|_{0,\infty} &\leq& \|R_hU^{1,0}-U_h^{1,0}\|_{0, \infty}+\|R_hU^{1,0}\|_{0, \infty} \nonumber \\
  &\leq& Ch^{-1}\|\theta^{1,0}\|_0 +\|R_hU^{1,0}\|_{0, \infty}\nonumber \\
  &\leq& CC'_0\tau h +\|R_hU^{1,0}\|_{0, \infty}\nonumber \\
   &\leq &K_1,
\end{eqnarray}
where $\tau<\tau_6$ and $h\leq h_0'$, such that $\tau h\leq 1/CC'_0$. For $n=1$, from \eqref{26} and \eqref{33}, one arrives at
\begin{eqnarray}\label{413}
 &&(\theta^1, v_h)+\frac{(\nu+i\eta)\tau}{2}(\nabla \theta^1, \nabla v_h)-\frac{\gamma\tau}{2}(\theta^1, v_h)\nonumber\\
   && = -(\lambda^1-\lambda^0, v_h)-\frac{(\nu+i\eta)\tau}{2}(\nabla \lambda^1+\nabla \lambda^0, \nabla v_h)
   -(\kappa+i\zeta)\tau(\mathcal{H}^1, v_h) \nonumber\\
  &&~~~~ +\frac{\gamma\tau}{2}(\lambda^1+\lambda^0, v_h),
\end{eqnarray}
where
\begin{equation*}
  \mathcal{H}^1:=f\bigg(\bigg|\frac{U^{1,0}+U^0}{2}\bigg|^2\bigg)\frac{U^{1,0}+U^0}{2}-f\bigg(\bigg|\frac{U_h^{1,0}+U_h^0}{2}\bigg|^2\bigg)\frac{U_h^{1,0}+U_h^0}{2}.
\end{equation*}
Setting $v_h=\theta^1$ in \eqref{413}, and taking the real part of the
resulting equation, we obtain
\begin{eqnarray}\label{414}
 &&\|\theta^1\|_0^2+\frac{\nu\tau}{2}\|\nabla \theta^1\|_0^2-\frac{\gamma\tau}{2}\|\theta^1\|_0^2\nonumber\\
   && = -Re(\lambda^1-\lambda^0, \theta^1)
   -Re\bigg\{(\kappa+i\zeta)\tau(\mathcal{H}^1, \theta^1)\bigg\}+\frac{\gamma\tau}{2}Re(\lambda^1+\lambda^0, \theta^1):=\sum_{j=1}^3\mathcal{K}_j,\nonumber\\
\end{eqnarray}
where
\begin{equation*}
  \mathcal{H}^1:=f\bigg(\bigg|\frac{U^{1,0}+U^0}{2}\bigg|^2\bigg)\frac{U^{1}+U^0}{2}-f\bigg(\bigg|\frac{U_h^{1,0}+U_h^0}{2}\bigg|^2\bigg)\frac{U_h^{1}+U_h^0}{2}.
\end{equation*}
As the result of \eqref{35}, it shows that
\begin{equation}\label{415}
 |\mathcal{K}_1|=\bigg|\tau Re\bigg(\frac{\lambda^1-\lambda^0}{\tau}, \theta^1\bigg)\bigg|
 \leq C\tau h^2\bigg\|\frac{U^1-U^0}{\tau}\bigg\|_2\|\theta^1\|_0\leq C\tau h^4+C\tau\|\theta^1\|_0^2.
\end{equation}
It is not difficult to derive that
\begin{equation}\label{416}
   |\mathcal{K}_2|=\bigg|Re\bigg\{(\kappa+i\zeta)\tau(\mathcal{H}^1, \theta^1)\bigg\}\bigg|
   \leq C\tau h^4+C\tau\|\theta^1\|_0^2+C\tau\|\theta^{1,0}\|_0^2,
\end{equation}
and
\begin{equation}\label{417}
  |\mathcal{K}_3|\leq C\tau h^4+C\tau\|\theta^1\|_0^2.
\end{equation}
Therefore, combining \eqref{415}-\eqref{417} with \eqref{414}, yields
\begin{equation}\label{418}
  \|\theta^1\|_0^2+\tau\|\nabla \theta^1\|_0^2\leq C\tau\|\theta^1\|_0^2+C\tau h^4.
\end{equation}
Hence, if $\tau$ is sufficiently small, we derive that
\begin{equation}\label{419}
  \|\theta^1\|_0+\sqrt{\tau}\|\nabla \theta^1\|_0\leq C'\sqrt{\tau}h^2.
\end{equation}
By the inverse inequality, we obtain
\begin{equation}\label{420}
  \|U_h^1\|_{0, \infty}\leq Ch^{-1}\|\theta^1\|_0+\|R_hU^1\|_{0, \infty}\leq K_1,
\end{equation}
where $\tau\leq \tau_7$ and $h\leq h_1$, such that $\sqrt{\tau}h\leq \sqrt{\tau_7}h_1\leq1/CC'$. In what follows, let us assume that \eqref{42}
holds for $m\leq n-1$, and one intends to prove its correctness
for $m=n$. By virtue of the assumption, we obtain
\begin{equation}\label{421}
  \|U_h^m\|_{0, \infty}\leq Ch^{-1}\|\theta^m\|_0+\|R_hU^m\|_{0, \infty}\leq K_1,\quad 0\leq m\leq n-1,
\end{equation}
where $h\leq h_2:=1/CC'$. It follows from \eqref{25} and \eqref{31} that
\begin{equation}\label{422}
  (\delta_t \theta^n, v_h)+(\nu+i\eta)(\nabla \widetilde{\theta}^n, \nabla v_h)-\gamma(\widetilde{\theta}^n, v_h)=-(\delta_t\lambda^n, v_h)
  -(\kappa+i\zeta)(\mathcal{H}^n, v_h)+\gamma(\widetilde{\lambda}^n, v_h),
\end{equation}
where
\begin{equation*}
 \mathcal{H}^n:=f(|\widehat{U}^n|^2)\widetilde{U}^n-f(|\widehat{U}^n|_h^2)\widetilde{U}_h^n.
\end{equation*}
Substitute $v_h=\widetilde{\theta}^n$ in \eqref{422}, and take the real part of the resulting equation to arrive at
\begin{eqnarray}\label{423}
  &&\frac{\|\theta^n\|_0^2-\|\theta^{n-1}\|_0^2}{2\tau}+\nu \|\nabla \widetilde{\theta}^n\|_0^2-\gamma\|\widetilde{\theta}^n\|_0^2  \nonumber\\
  &&= -Re(\delta_t\lambda^n, \widetilde{\theta}^n) -Re\big\{(\kappa+i\zeta)(\mathcal{H}^n, \widetilde{\theta}^n)\big\}
  +\gamma Re(\widetilde{\lambda}^n, \widetilde{\theta}^n):= \sum_{j=1}^3\mathcal{M}_j.
\end{eqnarray}
It is obvious from \eqref{23} that
\begin{equation}\label{424}
  |\mathcal{M}_1|\leq \|\delta_t\lambda^n\|_0\|\widetilde{\theta}^n\|_0
  \leq Ch^2\|\delta_tU^n\|_2\|\widetilde{\theta}^n\|_0\leq Ch^4\|\delta_tU^n\|_2^2+\frac{1}{4}\|\widetilde{\theta}^n\|_0.
\end{equation}
Thanks to \eqref{361}, we have
\begin{equation}\label{425}
  \|\delta_tU^n\|_2\leq \|\delta_tu^n\|_2+\bigg\|\frac{e^n-e^{n-1}}{\tau}\bigg\|_2
  \leq \|\delta_tu^n\|_2+2C_8\tau\leq K_2,
\end{equation}
where $\tau\leq \tau_8:=1/2C_8$. Therefore, from \eqref{424} and \eqref{425}, we obtain
\begin{equation}\label{426}
  |\mathcal{M}_1|\leq Ch^4+C(\|\theta^n\|_0^2+\|\theta^{n-1}\|_0^2).
\end{equation}
Utilize \eqref{23}, \eqref{35} and \eqref{421} to arrive at
\begin{eqnarray}\label{427}
  |\mathcal{M}_2| &=& \bigg|(\kappa+i\zeta)\bigg(f(|\widehat{U}^n|^2)\widetilde{U}^n-f(|\widehat{U}^n_h|^2)\widetilde{U}_h^n, \widetilde{\theta}^n\bigg)\bigg| \nonumber\\
    &\leq& \sqrt{\kappa^2+\zeta^2} \bigg|\bigg((f(|\widehat{U}^n|^2)-f(|\widehat{U}_h^n|^2))\widetilde{U}^n+f(|\widehat{U}^n_h|^2)(\widetilde{U}^n-\widetilde{U}_h^n), \widetilde{\theta}^n\bigg)\bigg| \nonumber\\
    &\leq& C(\|\theta^n\|_0^2+\|\theta^{n-1}\|_0^2+\|\theta^{n-2}\|_0^2)+Ch^4.
\end{eqnarray}
It follows from \eqref{23} that
\begin{equation}\label{428}
  |\mathcal{M}_3| = \big|\gamma Re(\widetilde{\lambda}^n, \widetilde{\theta}^n)\big|\leq
  Ch^4+C(\|\theta^n\|_0^2+\|\theta^{n-1}\|_0^2).
\end{equation}
Substituting \eqref{426}-\eqref{428} into \eqref{423} obtains that
\begin{equation}\label{429}
  \frac{\|\theta^n\|_0^2-\|\theta^{n-1}\|_0^2}{2\tau}+\nu \|\nabla \widetilde{\theta}^n\|_0^2\leq
  C(\|\theta^n\|_0^2+\|\theta^{n-1}\|_0^2+\|\theta^{n-2}\|_0^2)+Ch^4.
\end{equation}
Omitting the positive item $\nu \|\nabla \widetilde{\theta}^n\|_0^2$, and using the discrete Gronwall inequality, one concludes
\begin{equation}\label{430}
  \|\theta^n\|_0\leq C'h^2,
\end{equation}
which also implies
\begin{equation}\label{431}
  \|U_h^n\|_{0, \infty}\leq Ch^{-1}\|\theta^n\|_0+\|R_hU^n\|_{0, \infty}\leq K_1,
\end{equation}
where $h<h_1=1/CC'$.
By the mathematical induction, we have completed the proof of Theorem \ref{Theo2}. $\Box$

\end{proof}

\begin{Theo}\label{Theo3}
Under the conditions of Theorem \ref{Theo2}, the error estimate
\begin{equation}\label{432}
  \|\nabla (I_hu^n-U_h^n)\|_0\leq C(\tau^2+h^2)
\end{equation}
holds for $1\leq n\leq N$.
\end{Theo}

\begin{proof}
We first prove $\|\nabla \theta^n\|_0\leq \widetilde{C}h^2$ holds for $1\leq n\leq N$.
For the case of $n=1$, from \eqref{418}, the result is obviously right.
We now assume  $\|\nabla \theta^m\|_0\leq \widetilde{C}h^2$ holds for $m\leq n-1$, and intend to prove
$\|\nabla \theta^n\|_0\leq \widetilde{C}h^2$.

To this end, the error equation \eqref{422} is changed into the following one:
\begin{eqnarray}\label{433}
  \frac{1}{\nu+i\eta}(\delta_t \theta^n, v_h)+(\nabla \widetilde{\theta}^n, \nabla v_h)&=&\frac{\gamma}{\nu+i\eta}(\widetilde{\theta}^n, v_h)-\frac{1}{\nu+i\eta}(\delta_t\lambda^n, v_h)
  -\frac{\kappa+i\zeta}{\nu+i\eta}(\mathcal{H}^n, v_h)\nonumber\\
  &&+\frac{\gamma}{\nu+i\eta}(\widetilde{\lambda}^n, v_h),
\end{eqnarray}
Let us denote $v_h=\delta_t\theta^n$ in \eqref{433}, and take its real part to obtain
\begin{eqnarray}\label{434}
  &&\frac{\nu}{\nu^2+\eta^2}\|\delta_t \theta^n\|_0^2+\frac{\|\nabla \theta^n\|_0^2-\|\nabla\theta^{n-1}\|_0^2}{2\tau}\nonumber\\
  &&=Re\bigg\{\frac{\gamma}{\nu+i\eta}(\widetilde{\theta}^n, \delta_t\theta^n)\bigg\}-Re\bigg\{\frac{1}{\nu+i\eta}(\delta_t\lambda^n, \delta_t\theta^n)\bigg\}
  -Re\bigg\{\frac{\kappa+i\zeta}{\nu+i\eta}(\mathcal{H}^n, \delta_t\theta^n)\bigg\}\nonumber\\
  &&~~~+Re\bigg\{\frac{\gamma}{\nu+i\eta}(\widetilde{\lambda}^n, \delta_t\theta^n)\bigg\}:=\sum_{j=1}^4{\mathcal{N}_j}.
\end{eqnarray}
It is obvious that
%
\begin{eqnarray*}
   &&  |\mathcal{N}_1|\leq \frac{|\gamma|}{\sqrt{\nu^2+\eta^2}}\bigg(\epsilon_8\|\delta_t\theta^n\|_0^2+\frac{1}{4\epsilon_8}\|\widetilde{\theta}^n\|_0^2\bigg),  \\
  &&|\mathcal{N}_2|\leq  \frac{1}{\sqrt{\nu^2+\eta^2}}\bigg(\epsilon_9\|\delta_t\theta^n\|_0^2+\frac{1}{4\epsilon_9}\|\delta_t\lambda^n\|_0^2\bigg)
  \leq \frac{\epsilon_9}{\sqrt{\nu^2+\eta^2}}\|\delta_t\theta^n\|_0^2+Ch^4,  \\
  && |\mathcal{N}_3|\leq  \frac{\sqrt{\kappa^2+\zeta^2}}{\sqrt{\nu^2+\eta^2}}\epsilon_{10}\|\delta_t\theta^n\|_0^2+C\big(
  \|\theta^n\|_0^2+\|\theta^{n-1}\|_0^2+\|\theta^{n-2}\|_0^2\big)+Ch^4, \\
  &&  |\mathcal{N}_4|\leq  \frac{|\gamma|}{\sqrt{\nu^2+\eta^2}}\bigg(\epsilon_{11}\|\delta_t\theta^n\|_0^2+\frac{1}{4\epsilon_{11}}\|\widetilde{\lambda}^n\|_0^2\bigg)
  \leq \frac{|\gamma|\epsilon_{11}}{\sqrt{\nu^2+\eta^2}}\|\delta_t\theta^n\|_0^2+Ch^4.
\end{eqnarray*}
Then, substituting above inequalities into \eqref{434}, and taking
\begin{equation*}
  \epsilon_8=\epsilon_{11}=\frac{\sqrt{\nu^2+\eta^2}}{8|\gamma|}, \quad \epsilon_9=\frac{\sqrt{\nu^2+\eta^2}}{8},\quad \epsilon_{10}=\frac{\sqrt{\nu^2+\eta^2}}{8\sqrt{\kappa^2+\zeta^2}},
\end{equation*}
we arrive at
\begin{equation}\label{435}
  \frac{\|\nabla \theta^n\|_0^2-\|\nabla\theta^{n-1}\|_0^2}{2\tau}\leq C\big(\|\theta^n\|_0^2+\|\theta^{n-1}\|_0^2+\|\theta^{n-2}\|_0^2\big)+Ch^4.
\end{equation}
Then, by using the discrete Gronwall inequality in \eqref{435}, we have
\begin{equation}\label{436}
  \|\nabla \theta^n\|_0\leq \widetilde{C}h^2.
\end{equation}
Finally, with the help of \eqref{22}-\eqref{24}, we conclude
\begin{eqnarray}
  \|\nabla (I_hu^n-U_h^n)\|_0 &\leq& \|\nabla (I_hu^n-R_hu^n)\|_0 + \|\nabla (R_hu^n-R_hU^n)\|_0+\|\nabla (R_hU^n-U_h^n)\|_0 \nonumber\\
   &\leq& Ch^2\|u^n\|_3+C\|e^n\|_2 +C\|\nabla \theta^n\|_0 \nonumber\\
   &\leq& C(\tau^2+h^2).
\end{eqnarray}

\end{proof}

Based on Theorem \ref{Theo3} and the inerpolated postprocessing operator $I_{2h}^2$ \cite{Shi2016Unconditional},
the following global superconvergence result is deduced.
\begin{Theo}
Under the conditions of Theorem \ref{Theo2}, we have
\begin{equation}\label{438}
  \|u^n-I_{2h}^2U_h^n\|_1=O(\tau^2+h^2).
\end{equation}
\end{Theo}

\begin{Rem}
In this paper, the time-space error splitting technique is adopted to obtain the unconditional superconvergence results of a linearized Crank-Nicolson
Galerkin FEM for generalized Ginzburg-Landau equation. Due to the difference between the Ginzburg-Landau equations and the Schr{\"o}dinger equations,
the analysis procedure in this paper is different from one in \cite{shi2017unconditional}.
\end{Rem}

\section{Numerical Results}

\begin{example}\label{exa1}
Consider the NGLE
\begin{align}
& u_t - (1+i)\Delta u + (1+i)|u|^2u - u = g(X, t),\quad (X, t)\in \Omega\times (0, T],\label{51}\\
  &u(X, 0)=u_0(X),\quad X\in \Omega, \label{52}\\
  &u(X,t)=0,\quad (X, t)\in \partial \Omega\times (0, T], \label{53}
\end{align}
where $\Omega=[0, 1]\times [0, 1]$, and $g(X, t)$ is chosen corresponding to the plane wave solution
\begin{equation}\label{54}
  u(X, t)=e^{i(t-2x-2y)}xy(1-x)(1-y).
\end{equation}
\end{example}

In this example, a uniform rectangular partition with $M+1$ nodes in each direction is used in our computation, and the system \eqref{51}-\eqref{53} is numerically solved by the linearized Galerkin method with the bilinear element. On the one hand, taking $\tau=h$, and Tables 1-4 list the numerical
results with respect to the time $t=0.25, 0.5, 0.75, 1.0$. On the other hand, in order to show the unconditional stability, we choose
$h=1/80$, and the large time steps $\tau = h, 5h, 10h, 20h$, respectively. The coresponding results are shown in Table 5.
All the numerical results are concurring with the theoretical results.

 \begin{table}[!h]\label{table1}
\renewcommand\arraystretch{1.5}
 \tabcolsep 0pt
 \caption{Numerical results at $t=0.25$ with $\tau=h$}
 \vspace*{-10pt}
 \begin{center}
 \def\temptablewidth{1\textwidth}
 {\rule{\temptablewidth}{1pt}}
 \begin{tabular*}{\temptablewidth}{@{\extracolsep{\fill}}lllllll}
$M\times M$& $\|u^n-U_h^n\|_1$ & Order & $\|U_h^n-I_hu^n\|_1$ & Order & $\|u^n-I_{2h}^2U_h^n\|_1$ & Order\\   \hline
 $10\times 10$& 5.2551e-02 &- & 8.4988e-03 &-& 1.0060e-01& -\\
   $20\times 20$& 2.6409e-02 &0.9927 & 2.8420e-03  &1.5803& 3.8516e-02& 1.3852\\
   $40\times 40$& 1.3189e-02 &1.0016 & 5.8402e-04  &2.2828& 9.9039e-03& 1.9594\\
   $80\times 80$& 6.5963e-03 &0.9997& 1.4682e-04  &1.9920& 2.4934e-03& 1.9899
       \end{tabular*}
        {\rule{\temptablewidth}{1pt}}
        \end{center}
        \end{table}

 \begin{table}[!h]\label{table2}
\renewcommand\arraystretch{1.5}
 \tabcolsep 0pt
 \caption{Numerical results at $t=0.5$ with $\tau=h$}
 \vspace*{-10pt}
 \begin{center}
 \def\temptablewidth{1\textwidth}
 {\rule{\temptablewidth}{1pt}}
 \begin{tabular*}{\temptablewidth}{@{\extracolsep{\fill}}lllllll}
$M\times M$& $\|u^n-U_h^n\|_1$ & Order & $\|U_h^n-I_hu^n\|_1$ & Order & $\|u^n-I_{2h}^2U_h^n\|_1$ & Order\\   \hline
 $10\times 10$& 5.2622e-02 &- & 8.8880e-03 &-& 1.0004e-01& -\\
   $20\times 20$& 2.6354e-02 &0.9976 & 2.2875e-03  &1.9581& 3.8475e-02&1.3785 \\
   $40\times 40$& 1.3189e-02 &0.9986 & 5.8464e-04  &1.9682&9.9050e-03 & 1.9577\\
   $80\times 80$& 6.5963e-03 &0.9997& 1.4696e-04  &1.9922&2.4934e-03 &1.9900
       \end{tabular*}
        {\rule{\temptablewidth}{1pt}}
        \end{center}
        \end{table}

 \begin{table}[!h]\label{table3}
\renewcommand\arraystretch{1.5}
 \tabcolsep 0pt
 \caption{Numerical results at $t=0.75$ with $\tau=h$}
 \vspace*{-10pt}
 \begin{center}
 \def\temptablewidth{1\textwidth}
 {\rule{\temptablewidth}{1pt}}
 \begin{tabular*}{\temptablewidth}{@{\extracolsep{\fill}}lllllll}
$M\times M$& $\|u^n-U_h^n\|_1$ & Order & $\|U_h^n-I_hu^n\|_1$ & Order & $\|u^n-I_{2h}^2U_h^n\|_1$ & Order\\   \hline
 $10\times 10$& 5.2750e-02 &- & 9.5904e-03 &-&1.0008e-01 & -\\
   $20\times 20$& 2.6408e-02 &0.9982 & 2.8374e-03  &1.7570&3.8498e-02 & 1.3783\\
   $40\times 40$& 1.3189e-02 &1.0016 & 5.8462e-04  &2.2790&9.9045e-03 &1.9586 \\
   $80\times 80$& 6.5963e-03 &0.9997& 1.4695e-04  &1.9921&2.4934e-03 &1.9900
       \end{tabular*}
        {\rule{\temptablewidth}{1pt}}
        \end{center}
        \end{table}

\begin{table}[!h]\label{table4}
\renewcommand\arraystretch{1.5}
 \tabcolsep 0pt
 \caption{Numerical results at $t=1$ with $\tau=h$}
 \vspace*{-10pt}
 \begin{center}
 \def\temptablewidth{1\textwidth}
 {\rule{\temptablewidth}{1pt}}
 \begin{tabular*}{\temptablewidth}{@{\extracolsep{\fill}}lllllll}
$M\times M$& $\|u^n-U_h^n\|_1$ & Order & $\|U_h^n-I_hu^n\|_1$ & Order & $\|u^n-I_{2h}^2U_h^n\|_1$ & Order\\   \hline
 $10\times 10$& 5.2530e-02 &- & 8.3490e-03 &-& 9.9952e-02&- \\
   $20\times 20$& 2.6354e-02 &0.9951 & 2.2877e-03  &1.8677&3.8481e-02 &1.3771 \\
   $40\times 40$& 1.3189e-02 &0.9986 & 5.8462e-04  &1.9683& 9.9047e-03& 1.9580\\
   $80\times 80$& 6.5963e-03 &0.9997& 1.4695e-04  &1.9921& 2.4934e-03& 1.9900
       \end{tabular*}
        {\rule{\temptablewidth}{1pt}}
        \end{center}
        \end{table}

        \begin{table}[!h]\label{table5}
\renewcommand\arraystretch{1.5}
 \tabcolsep 0pt
 \caption{Convergence results of $\|u^n-U_h^n\|_1$ with $h=1/80$ and $\tau=kh$}
 \vspace*{-10pt}
 \begin{center}
 \def\temptablewidth{1\textwidth}
 {\rule{\temptablewidth}{1pt}}
 \begin{tabular*}{\temptablewidth}{@{\extracolsep{\fill}}lllll}
$t$& $k=1$ & $k=5$ & $k=10$ & $k=20$         \\    \hline
 0.25& 6.5963e-03 &6.5969e-03 & 9.5227e-03 &3.4159e-02\\
  0.50&6.5963e-03 &6.5968e-03 & 7.3813e-03 &2.6816e-02 \\
  0.75& 6.5963e-03 &6.5969e-03 & 6.8870e-03 &2.1424e-02 \\
   1.00& 6.5963e-03 &6.5968e-03 & 6.7265e-03 &1.8086e-02
       \end{tabular*}
        {\rule{\temptablewidth}{1pt}}
        \end{center}
        \end{table}

\bibliographystyle{spmpsci}
\bibliography{GLE}
\end{document}